\documentclass{amsart}

\usepackage{amsmath,amsfonts,amssymb,amsthm}
\usepackage{hyperref}
\usepackage{mathtools}
\usepackage{pdfpages}
 
\usepackage[capitalize]{cleveref}
\usepackage{stmaryrd}
\usepackage{tikz-cd}

\newtheorem{lem}{Lemma}[section]
\newtheorem{teo}[lem]{Theorem}

\newtheorem{pro}[lem]{Proposition}
\newtheorem{cor}[lem]{Corollary}

\newtheorem*{con*}{Conjecture}

\newtheorem{Conj}{Conjecture}
\newtheorem{Question}[Conj]{Question}

\theoremstyle{definition}
\newtheorem{exa}[lem]{Example}

\theoremstyle{remark}
\newtheorem{rem}[lem]{Remark}

\newcommand{\argu}{\hbox to 7truept{\hrulefill}}

\DeclareMathOperator{\im}{Im}

\DeclareMathOperator{\rg}{rg}

\newcommand{\myeq}[1]{\ensuremath{\stackrel{\text{#1}}{=}}}

\newcommand{\myge}[1]{\ensuremath{\stackrel{\text{#1}}{\geqslant}}}

\newcommand {\hrg}{{\rm Hrg}}
\newcommand {\nhrg}{{\rm Hrg^{\unlhd}}}

\newcommand{\Z}{\mathbb{Z}}
\newcommand{\F}{\mathbb{F}}

\newcommand{\N}{\mathbb{N}}
\newcommand{\CC}{\mathbb{C}}
\newcommand{\Q}{\mathbb{Q}}

\renewcommand{\epsilon}{\varepsilon}

\newcounter{steffencomments}

\newcounter{andreicomments}

 \date{\today}
 
\begin{document}
\title[Invariants of just infinite groups] {Asymptotic invariants of residually finite just infinite groups
\newline \newline
\rm{Dedicated to the memory of Nikolay A. Vavilov}}
\author{Andrei Jaikin-Zapirain}
 \address{Departamento de Matem\'aticas, Universidad Aut\'onoma de Madrid \and  Instituto de Ciencias Matem\'aticas, CSIC-UAM-UC3M-UCM}
\email{andrei.jaikin@icmat.es}

\author{Steffen Kionke}
 \address{FernUniversit\"at in Hagen, Fakult\"at für Mathematik und Informatik, Universit\"atsstr. 11, Hagen}
 \email{steffen.kionke@fernuni-hagen.de}
\begin{abstract}
Recently, Eduard Schesler and   the second author constructed examples of finitely generated   residually finite, hereditarily just infinite groups with positive first $L^2$-Betti 
number. In contrast to their result, we show that a finitely generated residually-$p$ just infinite group has trivial first $L^2$-Betti number. Moreover, we prove that the normal homology rank gradient of a finitely generated, residually finite, just infinite group vanishes. \end{abstract}

\maketitle

\section{Introduction}  
The study of normal subgroups is a classical topic in group theory, and a large part of Nikolay Vavilov's research was dedicated to investigating the normal subgroup structure of Chevalley and Steinberg groups over arbitrary commutative rings (see \cite{VG04, VS08, Va11} and references therein). Here, we consider just infinite groups, i.e. infinite groups whose all non-trivial normal subgroups are of finite index. The importance of this family of groups lies in the fact that, by Zorn's Lemma, every finitely generated  infinite  group has a just infinite quotient. If a just infinite group is not residually finite, then it  is virtually a power of a simple group.

If we measure the height of a group as the Krull dimension of the lattice of its normal subgroups, then free groups will be among the tallest  and just infinite groups among the shortest. Fixing a natural number \(d\), a natural measure of fatness for a \(d\)-generated group \(\Gamma\) is its first \(L^2\)-Betti number \(b_1^{(2)}(\Gamma,\Q) \). Notice that for a \(d\)-generated group \(\Gamma\), \(b_1^{(2)}(\Gamma,\Q)  \le d-1\) and, moreover, if \(d > 1\), \(b_1^{(2)}(\Gamma,\Q)  = d-1\) if and only if \(\Gamma\) is free. Therefore, it is quite surprising that for any \(\epsilon > 0\) and a natural number \(d > 1\), there exists a simple \(d\)-generated group \(\Lambda\) with \(b_1^{(2)}(\Lambda,\Q)  > d-1-\epsilon\) \cite{OT11} and there exists a residually finite just infinite \(d\)-generated group \(\Gamma\) with \(b_1^{(2)}(\Gamma,\Q)  > d-1-\epsilon\) \cite{KS24}. 
In our first result we show that this is impossible in the case of finitely generated residually-$p$ just infinite  groups.

\begin{teo}\label{main}
Let $\Gamma$ be a finitely generated residually-$p$ just infinite  group. Then
$b_1^{(2)}(\Gamma,\Q)=0$.
 \end{teo}
  A group $\Gamma$ is called  locally zero-one if every finitely
generated subgroup is either finite or of finite index. In \cite{EJ13} Ershov and  the first author constructed  examples of  infinite, finitely generated  residually finite and locally zero-one $p$-groups.  
  \begin{cor} Let $\Gamma$ be a  residually finite and locally zero-one $p$-group.  
Then $b_1^{(2)}(\Gamma,\Q)=0$.
  \end{cor}
 \begin{proof}
If $\Gamma$ is locally finite, then it is amenable, and so $b_1^{(2)}(\Gamma,\Q)=0$.
Assume that $\Gamma$ is not locally finite. Then it is finitely generated, and so it has a just infinite quotient $\Gamma/N$. Since $\Gamma$ is locally zero-one, $N$ is locally finite, and so amenable.

 If $N$ is finite, then since $\Gamma$ is residually finite, $\Gamma $ and $\Gamma/N$ are commensurable. By Theorem \ref{main}, $b_1^{(2)}(\Gamma/N,\Q)=0$. Therefore $b_1^{(2)}(\Gamma,\Q)=0$.
 
Now assume that $N$ is infinite.  Since it is amenable, $b_1^{(2)}(N,\Q)=0$. By a result of Gaboriau~\cite[Theorem~6.8]{Ga02} (see also~\cite{Pe24} for a recent new proof), a group with an infinite normal subgroup of infinite index and finite first $L^2$-Betti number has trivial first $L^2$-Betti number. Therefore, $b_1^{(2)}(\Gamma, \mathbb{Q})$ must be zero.
 \end{proof}
 
The proof of Theorem \ref{main} has two main ingredients. Let $\Gamma_{\hat p}$ be the pro-$p$ completion of $\Gamma$. Using a variation of results of Barnea, Schlage-Puchta \cite[Theorem 3.1]{BSP13} and Pappas \cite[Theorem 3.2]{Pa13} we show that if the rank gradient $\rg(\Gamma_{\hat p})$ of $\Gamma_{\hat p}$ is positive,  then the quotient of $\Gamma$ by the normal subgroup generated by a large $p$-power of an element of $\Gamma$ is infinite. In particular,   $\Gamma$ cannot be just infinite. 
On the other hand, by a result  of Erhov and L\"uck \cite[Theorem 1.6]{EL14}, $b_1^{(2)}(\Gamma,\Q)\le \rg(\Gamma_{\hat p})$.  This is explained in more detail in Section \ref{prop}.
In Section \ref{appendix} we provide new proofs of the results from \cite{BSP13, Pa13, EL14} mentioned above, which we hope will interest the reader.

The notion of rank gradient of pro-$p$ groups can be generalized in different ways to arbitrary profinite groups $G$. We consider the following generalization, called 
the (lower) normal homology rank gradient of $G$
$$\nhrg (G)= \liminf_{U\unlhd_o G}\frac{d({\rm H}_1(U,\widehat \Z))}{|G:U|}=  \sup_{V\unlhd_o G}\left (\inf _{U\unlhd_o G, U\le V}\frac{d({\rm H}_1(U,\widehat \Z))}{|G:U|}\right ).$$

In our second result, proved in Section \ref{homology}, we show that for a finitely generated residually finite just infinite group, the normal homology rank gradient of its profinite completion vanishes.
\begin{teo}\label{teo:nhrg-just-infinite}
Let $\Gamma$ be a finitely generated dense just infinite subgroup of a profinite group $G$. 
Then $\nhrg(G) =0$.
\end{teo}
 
 A group is called hereditarily just infinite if all its subgroups of finite index are just infinite. We finish the Introduction by proposing the following natural question that arises from our work.
 \begin{Question} Is there a finitely generated hereditarily just infinite group with positive first $L^2$-Betti number which is residually-$P$ for a finite set of primes $P$?
 \end{Question}
 
 \section*{Acknowledgments}
The work of the first author is partially supported by the grant PID2020-114032GB-I00/AEI/ 10.13039/501100011033.

We are grateful to Pablo S\'anchez-Peralta, Henrique Souza and anonymous referees  for their useful comments on preliminary versions of this paper.
 \section{Residually-$p$ just infinite groups have trivial first $L^2$-Betti number} \label{prop}
Let $G$ be a finitely generated pro-$p$ group and $G>U_1>U_2>\ldots $ a residual chain (i.e. $\displaystyle \cap_i U_i=\{1\}$) of normal open subgroups of $G$. The minimal number of generators $d(G)$ of $G$ coincides with $\dim_{\F_p} {\rm H}_1(G;\F_p)$ and by the Nielsen-Schreier formula $d(U_i)-1\le |G:U_i|(d(G)-1)$. Therefore, the sequence $\frac{d(U_i)}{|G:U_i|}$ is decreasing and bounded from below. Hence,  we can define the limit
$$\rg (G)=\lim_{i\to \infty} \frac {d(U_i)}{|G:U_i|},$$
which is called the rank gradient of $G$. It does not depend on the chain and, in fact, we also have
\begin{equation}\label{infimum2}
\rg(G)=\inf_{U\lhd_o G} \frac{d(U)}{|G:U|}.
\end{equation}

Observe that if $\rg(G)>0$, then $G$ is infinite.

The following result is a variation of results of Barnea, Schlage-Puchta \cite[Theorem 3.1]{BSP13} and Pappas \cite[Theorem 3.2]{Pa13}. In the appendix we will present  a new proof of it.
\begin{pro}\label{rankgradient}
Let $G$ be a finitely generated pro-$p$ group,  $g\in G$ and $k\in \N$. Let  $N$ be the closed normal subgroup  of $G$ generated by $g^{p^k}$ and $\overline G=G/N$. Then
$\rg(\overline G)\ge \rg(G)-\frac 1 {p^k}$.
\end{pro}

For the definition of $L^2$-Betti numbers of groups we refer the reader to the book of L\"uck \cite{Luc02}. The definition of the first $L^2$-Betti number is also presented in the appendix.
The following result was proved by Ershov and L\"uck \cite[Theorem 1.6]{EL14}. In the appendix we will provide an alternative proof of it.

\begin{pro} \label{EL}
Let $\Gamma$ be a finitely generated residually-$p$ group. Then $b_1^{(2)}(\Gamma, \Q)\le \rg (\Gamma_{\hat p})$.
\end{pro}

Now we are ready to prove Theorem \ref{main}.

\begin{proof}[Proof of  Theorem \ref{main}]
Let $\Gamma$ be a finitely generated residually-$p$ just infinite group. Assume that $b_1^{(2)}(\Gamma, \Q)>0$. Then by Proposition \ref{EL}, $\alpha=\rg(\Gamma_{\Hat p})>0$. Let $k$ be such that $\frac 1 {p^k}<\alpha$.
By Zel'manov's solution of the restricted Burnside problem \cite{Ze90, Ze91}, we know that there exist $g\in \Gamma$ such that $o(g)>p^k$ (note that, since $\Gamma$ is residually $p$, $o(g)$ is either infinite or a power of $p$). By Theorem \ref{rankgradient}, the normal subgroup of $\Gamma$ generated by $g^{p^k}$ is of infinite index. But this contradicts the condition that $\Gamma$ is just infinite.
\end{proof}

\section{Just infinite groups have vanishing normal homology rank gradient}\label{homology}
Let $G$ be a finitely generated profinite group. Observe that $${\rm H}_1(G;\widehat \Z)\cong G/\overline{[G,G]}.$$ Define the   homology rank gradient and normal homology rank gradient of $G$ as
$$\hrg (G)= \liminf_{U\le_o G}\frac{d({\rm H}_1(U;\widehat \Z))}{|G:U|} \textrm{\ and\ } \nhrg (G)= \liminf_{U\unlhd_o G}\frac{d({\rm H}_1(U;\widehat \Z))}{|G:U|}.$$
We note that $\hrg(G) \leq \nhrg(G)$ and, moreover, for $H \unlhd_o G$ we have $$\hrg(H) =  |G:H| \hrg(G).$$

 For a finitely generated abstract group $\Gamma$, its first homology rank gradient and normal homology rank gradient are  defined to be the  homology rank gradient  and the normal homology rank gradient of its profinite completion 
$$\hrg (\Gamma)=\hrg(\widehat \Gamma), \ \nhrg (\Gamma)=\nhrg(\widehat \Gamma).$$

If   $G$ is a pro-$p$ group and $U$ is an  open subgroup, there exists an open normal subgroup $K$ of $G$ contained in $U$; then one has 
$$\frac{\dim_{\F_p} {\rm H}_1(K;\F_p)}{|G:K|}\le \frac{\dim_{\F_p} {\rm H}_1(U;\F_p)}{|G:U|}.$$
Therefore,  this implies that in this case $\nhrg(G)=\hrg(G)=\rg(G)$.

The proof of Theorem \ref{teo:nhrg-just-infinite} relies on the following two simple lemmas.
\begin{lem}\label{lem:1-generated}
Let $G$ be a profinite group. Let $K$ be the closed normal subgroup generated by $g \in G$.
Then $d(K/[K,K]) \leq |G:C_G(g)K|$, where $C_G(g)$ denotes the centralizer of $g$ in $G$.
\end{lem}
\begin{proof}
If $T$ is a transversal of $C_G(g)K$ in $G$, then $K/[K,K]$ is generated by $\{g^t:t\in T\}$. The assertion follows, since $|T| =  |G:C_G(g)K|$.
\end{proof}

\begin{lem}\label{lem:bound-irreps}
Let $Q$ be a finite group of exponent $e$ and let $p$ be a prime.
The degree of any simple $\F_{p}[Q]$-module is at most $\sqrt{e\cdot |Q|}$.
\end{lem}
\begin{proof}
Let $J$ be the Jacobson radical of  $\F_p[Q]$. Then $$\F_p[Q]/J\cong  \prod_{i=1}^r M_{n_i}(F_i)$$ where $F_i$ is a finite extension of $\F_p$ of degree $e_i$. The $i$th factor corresponds to a simple $\F_p[Q]$-module of degree $n_ie_i$ (and all simple $\F_p[Q]$-modules are obtained in this way). We have $\sum_{i=1}^r n_i^2 e_i \leq |Q|$. Adjoining to $\F_p$ all the $e$th roots of unity gives a field $F$. It a splitting field for $Q$: all fields $F_i$ are subfields of $F$. Hence $e_i \leq e$ for all $i$ and thus
\[
	(n_ie_i)^2 \leq \sum_{i=1}^r n_i^2 e_i e \leq e\cdot |Q|. \qedhere
\]
\end{proof}

Let $C \geq 1$. We say that a natural number $n$ is a $C$-almost prime if for every factorization $n=kl$, either $k\le C$ or $l\le C$. It is not difficult to see that if $n>1$ is  a $C$-almost prime, then there exists a prime $p$ dividing $n$ such that $\frac np\le C^2$ (write $n=kl$ with $k<l$ and $k$ maximal, then every prime divisor of $l$ has this property). 

\begin{proof}[Proof of Theorem \ref{teo:nhrg-just-infinite}] We assume for a contradiction that $\nhrg(G)>0$. Since $\Gamma$ is just infinite and dense in $G$, every open normal subgroup $N$ of $G$ contains an open normal subgroup $K$ normally generated by a single element. By Lemma \ref{lem:1-generated}, $d(K/[K,K]) \leq |G:K|$. Hence  $\nhrg(G)\leq 1$.  Let $\alpha=\frac 2{\nhrg (G)}$.

 From the definition of  homology rank gradient, it follows that there exists an open normal subgroup $U$ of $G$ such that 
 \begin{equation}\label{conditions}
 |G:U|\ge \frac{4\alpha} 3 \textrm{\ and\ }
 \frac{d(N/[N,N])}{|G:N|}\ge \frac 34\nhrg(G)=\frac 3{2\alpha} \end{equation}
 for every open normal subgroup $N$ of $G$ contained in $U$ and therefore
$$d(N/[N,N])\ge   \frac {3|G:N|}{2\alpha}\ge 2.$$

\medskip

\textit{Step 1:}   Without loss of generality we can assume that $U$ is prosoluble.

Since for every open normal subgroup $N$ of $G$ contained in $U$, $d(N/[N,N])\ge 1$,   the maximal prosoluble quotient $ U/S$ of $U$ is infinite. The subgroup $S$ is characteristic in $U$, and so it is normal in $G$. Since $\Gamma$ is just infinite, it intersects $S$ trivially and also embeds into $ G/S$. Moreover, since the profinite group  $S$ does not have continuous abelian quotients,  by (\ref{conditions}),  $\nhrg ( G/S)\ge  \frac{3}{4}\nhrg(G)>0$. Thus, without loss of generality we will   assume that $S=1$ and so $U$ is prosoluble.

\medskip

\textit{Step 2:} For every non-trivial element $g\in U\cap \Gamma$  the order $o(g)$ of $g$ is $\alpha$-almost prime.

Assume that $o(g)$ is infinite  or is not $\alpha$-almost prime. Then there exists  a finite quotient $G/T$ of $G$ such that  the order $o(\overline g)$ of the image $\overline g$ of $g$ in $G/T$ is not $\alpha$-almost prime.  Let $o(\overline g)=k\cdot l$ with, $k,l>\alpha$. 

The group $UT/T$ is finite soluble. Therefore, there exists a normal subgroup $Q=M/T\le UT/T$ of $G/T$ such that $Q\cap\langle \overline g\rangle $ is of prime order (say $p$). Thus  $p$ divides $k$ or $l$, and so, $\frac {o(\overline g)}p>\alpha$. 

Consider the normal closed subgroup $K$ of $G$ generated by $g^{\frac {o(\overline g)}p}$. Since $\Gamma$ is just infinite, the subgroup $K$ is open in $G$. Moreover $K\le M$, and so the order of the image of $g$ in $G/K$ is $\frac {o(\overline g)}p$.  Since $g$ centralizes $g^{\frac {o(\overline g)}p}$, Lemma~\ref{lem:1-generated} gives $d(K/[K,K])\leq  \frac {|G:K|}{\alpha}$. On the other hand, $K\le U$ implies  $d(K/[K,K])\ge \frac {3|G:K|}{2\alpha}$. This gives a contradiction and hence
$o(g)$ is   $\alpha$-almost prime.

\medskip

 \textit{Step 3:}  Let $ P$ be a finite set of primes. Then the maximal pro-$P$ quotient of $G$ is finite.

Denote by $\overline G$  the maximal pro-$P$ quotient of $G$ and assume that it is infinite. Since $\Gamma$ is just infinite, we can see $\Gamma$ as a dense subgroup of $\overline G$. In particular, all elements of $\Gamma$ are $P$-elements. By the solution of the restricted Burnside problem, we know that there exist $g\in \Gamma$ such that $o(g)$ is arbitrarily large. However, this contradicts the fact that $o(g)$ should also be $\alpha$-almost prime by Step 2.

\medskip

\textit{Step 4:} Let $\epsilon > 0$. There is an open normal subgroup $U_\epsilon \unlhd_o G$ contained in $U$
such that $G/N$ has exponent at most $\epsilon^2 |G:N|$ for all $N \unlhd_o G$ contained in $U_\epsilon$.

By Step 2, $\Gamma$ is a torsion group. Therefore,  it follows from Step 3, that there is a prime $p$ with $\frac{1}{p} < \epsilon^2$ and an element $g \in \Gamma\cap U$ of order $p$.
Let $K$ be the closed normal subgroup generated by $g$. Since $\Gamma$ is just infinite, $K$ is open and $K/[K,K]$ is a finite elementary abelian $p$-group of order at least $p^{\frac{3|G:K|}{2\alpha}} \geq p^2$.
Define $U_\epsilon = [K,K]$. Let $N \subseteq U_\epsilon$ be open normal in $G$, then the exponent of $G/N$ is at most $$p|G/K||U_\epsilon:N| \leq \frac{|G/N|}{p} \leq \epsilon^2 |G:N|.$$

\medskip
 
 \textit{Step 5:} If $N \unlhd_o G$ is contained in $U$, then the exponent of $N/[N,N]$ is $\alpha$-almost prime. 

This is immediate from Step 2 using that $\Gamma$ is dense in $G$.

\medskip

\textit{Step 6:} There is an open normal subgroup with small homology rank.

Let $\epsilon > 0$ be given.
By Step 2,  $\Gamma$ is a torsion group. Thus, by Step 3, there is an element $g \in \Gamma \cap U_\epsilon$ of prime order $p > \alpha$. Let $K$ be the closed normal subgroup of $G$ generated by $g$. Since $\Gamma$ is just infinite, $K$ is open and $K/[K,K]$ is a finite elementary abelian $p$-group.

Define $P = \{q \text{ prime } \mid q \leq \alpha\} \cup \{p\}$.  By Step 3, the maximal pro-$P$ quotient $K/M_0$ of $K$ is finite. 
Since $p$ divides the order of $K/M_0$ and $K/M_0$ is soluble,
   there are normal subgroups $M_0\le M\le A\le K$ of $G$, such that $\overline{A}=A/M$ is a non-trivial  $p$-elementary abelian group and it is also a simple $\F_p[G/A]$-module. Among all such pairs $M\le A$ choose one with minimal $M$.
   
   Since the order of $\overline{A}$ is divisible by $p$ and the exponent of $A/[A,A]$ is $\alpha$-almost prime, it follows that $A/[A,A]$ is a $P$-group. Therefore, $M_0 \subseteq [A,A]$. We can find $[A,A] \leq A_0 \leq A$ such that $A_0/[A,A]$ is a non-trivial $p$-elementary abelian group and also a simple $\mathbb{F}_p[G/A]$-module. 
However, since $[A,A] \leq M$ and $M$ is minimal, it follows that $M = [A,A]$. Therefore, $A/[A,A] \cong \overline{A}$.

By Lemma \ref{lem:bound-irreps} and Step 4, we have $\dim_{\F_p} \overline{A} \leq \epsilon |G:A|$.
Since $\epsilon$ was arbitrary and $d(A/[A,A])/|G:A| \leq \epsilon$, this completes the proof.
\end{proof}

We note that Theorem \ref{teo:nhrg-just-infinite} implies:
\begin{cor}
Let $\Gamma$ be a finitely generated, residually finite, just infinite group.  Then $\nhrg (\Gamma)=0$.
\end{cor}
\begin{rem}
The $d$-generated, residually finite, just infinite groups constructed in \cite{KS24} have large ``upper normal homology rank gradient'', i.e.,
\[
	\limsup_{N\unlhd_o \Gamma}\frac{d({\rm H}_1(N;\Z))}{|G:N|} > d-1 -\epsilon.
\]
In particular, the ``lower'' and ``upper''  normal homology rank gradient can be arbitrarily far apart. 
\end{rem}
\begin{rem}
Under the assumptions of Theorem \ref{teo:nhrg-just-infinite}, the weaker conclusion $\hrg(G) =0$ can  be obtained using the following simple argument:

First observe that a residually finite just infinite group that is not hereditarily just infinite virtually decomposes as a direct product of two infinite groups and hence cannot have positive normal homology rank gradient (see e.g.~\cite[Lemma 2.9]{KS24}).
Therefore,  $\Gamma$ is  hereditarily just infinite.

Assume that $\hrg (G)>0$. If $H$ is an open subgroup of $G$, then 
\begin{equation}\label{eq:mult-hrg}
\hrg(H)=  |G:H|\hrg(G).
\end{equation}  Since  $\Gamma$ is hereditarily just infinite, it is enough to prove the theorem for $H\cap \Gamma$. In particular, without loss of generality we can assume that
$\hrg(G)>2$. From the definition of  homology rank gradient, it follows that there exists an open subgroup $U$ of $G$ such that for every open normal subgroup $K$ of $U$, $d({\rm H}_1(K,\widehat \Z ))\ge \frac {3|G:K|}2$.  Let $1\ne g \in U\cap  \Gamma$. Lemma~\ref{lem:1-generated} applied to the normal closed subgroup $K$ of $G$ generated by $g$ gives $d(K/[K,K] ) \leq |G:K|$. This gives a contradiction. 

We don't know whether the normal homology rank gradient satisfies an inequality analogous to \eqref{eq:mult-hrg}.
\end{rem}

\section{Appendix: new proofs of Proposition \ref{rankgradient} and Proposition \ref{EL}} \label{appendix}
Our proofs of Propositions \ref{rankgradient} and \ref{EL} are based on the observation that for a finitely generated group $\Gamma$, $\rg(\Gamma_{\hat{p}})$ and $\beta_1(\Gamma, \mathbb{Q})$ can be expressed  in terms of  ranks of the module $I_{\mathbb{Z}[\Gamma]}$. Therefore, we will first present the general theory of Sylvester module rank functions   main results about them.
  \subsection {Sylvester rank functions}	\label{sylvester}
In this paper all  rings are unitary and $R$-modules are left $R$-modules.
A {\bf Sylvester module rank function} $\dim$ on a   ring $R$ is a function that   assigns a non-negative real number to each finitely presented $R$-module   and satisfies the following conditions.
  \begin{enumerate}
\item [(SMod1)] $\dim \{0\} =0$, $\dim R =1$;
\item [(SMod2)]  $\dim(M_1\oplus M_2)=\dim M_1+\dim M_2$;
\item[(SMod3)] if $M_1\to M_2\to M_3\to 0$ is exact then
$$\dim M_1+\dim M_3\ge \dim M_2\ge \dim M_3.$$
\end{enumerate}
The space of Sylvester module rank functions on $R$ is denoted by $\mathbb{P}(R)$. If $\dim$, $\dim_i\in \mathbb{P}(R)$  $(i\in \N)$ are Sylvester module rank functions  then we write $$\dim=\displaystyle \lim_{i\to \infty} \dim_i \text{in the space }\mathbb{P}(R)$$ if   for every finitely presented $R$-module $M$, $\dim M=\displaystyle  \lim_{i\to \infty} \dim_i M$. 

For two Sylvester module rank functions $\dim_1$ and $\dim_2$ on $R$ we write $\dim_1\le \dim _2$ if for every finitely presented 
 $R$-module $M$, $\dim_1 M\le \dim_2 M$.
 
One  can extend  Sylvester module rank functions $\dim$ on $R$  to an arbitrary
finitely generated $R$-module $M$  in the following natural way:
\begin{equation} \label{liext}
\dim M=\inf \{\dim \widetilde M:\ \widetilde M \textrm{\ is   finitely presented   and\ } M \textrm{\ is a quotient of\ }\widetilde M\}.\end{equation}
It is easy to verify that the properties (SMod2) and (SMod3) still hold true.  Li \cite{Li21} constructed an extension of $\dim$ to all $R$-modules, but we will not use it in this paper.
 
 One should be cautious. The existence of the limit $\dim=\displaystyle \lim_{i\to \infty} \dim_i$ in $\mathbb{P}(R)$ does not imply that $\dim M=\displaystyle  \lim_{i\to \infty} \dim_i M$ for any finitely generated $R$-module $M$.  
 If $M$ is a finitely generated $R$-module, we  always have that  
 \begin{equation}\label{easyin}
 \dim M \ge   \limsup_{i\to \infty} \dim_i M.
 \end{equation}
 However in \cite{Ja24}, the following result was proved.
 \begin{pro} \label{limitfg}Let $R$ be a ring and let $\dim$, $\dim_i\in \mathbb{P}(R)$ $(i\in \N)$ be  Sylvester module rank functions on $R$. Assume that $\dim=\displaystyle \lim_{i\to \infty} \dim_i$ and for all $i$, $\dim\le \dim_i$. Then,  for any finitely generated $R$-module $M$, $$\dim M=\displaystyle  \lim_{i\to \infty} \dim_i M.$$
\end{pro}
\begin{exa}\label{example}
Here are some examples of Sylvester module rank functions that will appear in this paper. 
\begin{enumerate}
\item[(a)] Let $\Gamma$ be a group and $K$ a field. Then for any normal subgroup $N$ of finite index in $\Gamma$ and any finitely presented $K[\Gamma]$-module $M$, we can define $$\dim_{K[\Gamma/N]}M=\frac{\dim_K K\otimes_{K[N]} M}{|\Gamma:N|}.$$

\item[(b)] Let $f:R\to S$ be a homomorphism of rings and $\dim \in \mathbb P(S)$. Then we can define the induced Sylvester module rank function $\dim^f\in  \mathbb P(R)$: for every finitely presented $R$-module $M$, $\dim^f M=\dim S\otimes_R M$. Normally $f$ will be clear from the context, and therefore we will write $\dim$ instead of $\dim^f $. For example, in this way the function 
$\dim_{K[\Gamma/N]}$ from part (a) induces the corresponding  Sylvester module rank function on $K[\Gamma]$. 

\item [(c)] Let $\Gamma$ be a residually finite group, $K$ a field of characteristic 0 and $\Gamma>N_1>N_2>\ldots$ a residual chain  of normal subgroups of 
finite index. Then by L\"uck's approximation (\cite[Theorem 0.1]{Lu94}, \cite[Theorem 1.3]{Ja19}) there exists $\lim_i \dim_{K[\Gamma/N_i]}\in \mathbb P(K[\Gamma])$, which is independent on the chain. We denote this limit by $\dim_{K[\Gamma]}$. When $K$ is a subfield of $\CC$, then $\dim_{K[\Gamma]}$ is called the von Neumann dimension of $K[\Gamma]$-modules. 

\item[(d)] In (c), we have defined the von Neumann dimension $\dim_{K[\Gamma]}$ ($K$ is a subfield of $\CC$) only for residually finite groups, but it can be defined more generally (see, for example, \cite[Chapter 1]{Luc02}). This is done by considering the von Neumann group algebra $\mathcal{N}(\Gamma)$, which possesses a canonical dimension function $\dim_{\mathcal{N}(\Gamma)} \in \mathbb{P}(\mathcal{N}(\Gamma))$. Then the Sylvester function $\dim_{K[\Gamma]}$ is induced by the embedding $K[\Gamma] \hookrightarrow \mathcal{N}(\Gamma)$. \end{enumerate}
\end{exa}

Let  $G$ be a   pro-$p$ group. We will now introduce another Sylvester module rank function $\dim_{\F_p[[G]]}$.  If $U$ is a normal open subgroup of $G$, then as in the example (a), we can define $\dim_{\F_p[G/U]}\in \mathbb P(\F_p[[G]])$. 

\begin{lem}\cite[Lemma 4.1]{BLLS14} \label{decreasing}
Let $G$ be a   pro-$p$ group and $U_1\le U_2$ two  normal open subgroup of $G$.   Then $\dim_{\F_p[G/U_1]}\le \dim_{\F_p[G/U_2]}$ viewed as elements of $\mathbb P(\F_p[[G]])$
\end{lem}
Let $G>U_1>U_2>\ldots$ be a residual chain of normal open subgroups of $G$. By Lemma \ref{decreasing}, there exists $\lim_i \dim_{\F_p[G/U_i]}\in \mathbb P(\F_p[[G]])$, which is independent on the chain. We denote this limit by $\dim_{\F_p[[G]]}$. An alternative way to define $\dim_{\F_p[[G]]}$ is by using infimum: for a finitely presented  $\F_p[[G]]$-module $M$ we have that
\begin{equation}\label{infimum}
\dim_{\F_p[[G]]} M=\inf_{U\lhd_o G} \dim_{\F_p[G/U]} M.
\end{equation}
Combining Lemma \ref{decreasing} with Proposition \ref{limitfg} we also obtain the following result.
\begin{pro}\label{modpfg}
Let $G$ be a   pro-$p$ group and $G>U_1>U_2>\ldots$  a residual chain of normal open subgroups of $G$.  Let  $M$  be a finitely generated $\F_p[[G]]$-module and $N$   a closed normal subgroup of  $G$. Then
\begin{enumerate}
\item $\dim_{\F_p[[G]]} M=\displaystyle \lim_{i\to \infty} \dim_{\F_p[G/U_i]} M$ and
\item $\dim_{\F_p[[G]]} M \le \dim_{\F_p[[G/N]]} M$, where  $\dim_{\F_p[[G]]} $ and $ \dim_{\F_p[[G/N]]}$ represent elements of  $\mathbb P(\F_p[[G]]).$
\end{enumerate}
 \end{pro}

Given a finitely generated group $\Gamma$, let $I_{\Z[\Gamma]}$ denote the augmentation ideal of $\Z[\Gamma]$. 
The  first $L^2$-Betti number of $\Gamma$ is defined as $$b_1^{(2)}(\Gamma, \Q)=\dim_{\mathcal N(\Gamma)}{\rm H}_1(\Gamma; \mathcal N(\Gamma))=\dim_{\mathcal N(\Gamma)} (\mathcal N(\Gamma)\otimes_{\Z[\Gamma]} I_{\Z[\Gamma]}) -1+\frac 1{|\Gamma]}.$$ In view of  Example \ref{example} (d) we have an equivalent definition
\begin{equation} \label{l2} b_1^{(2)}(\Gamma, \Q)=\dim_{\Q[\Gamma]} I_{\Z[\Gamma]}-1+\frac 1{|\Gamma]}.\end{equation}
Let $\Gamma_{\hat p}$ be the pro-$p$ completion of $\Gamma$. We can also consider $\dim_{\F_p[[\Gamma_{\hat p}]]}$ as an element of $\mathbb P(\Z[\Gamma])$, and if $\Gamma$ is residually-$p$,  define the first mod-$p$ $L^2$-Betti number of $\Gamma$ as
$$b_1^{(2)}(\Gamma, \F_p)=\dim_{\F_p[[\Gamma_{\hat p}]]}I_{\Z[\Gamma]}-1+\frac 1{|\Gamma]}.$$

\subsection{Gradient of pro-$p$ completion and the  first mod-$p$ $L^2$-Betti number}

For a pro-$p$ group $G$, let $I_{\F_p[[G]]}$ denote the augmentation ideal of the completed group algebra $\F_p[[G]]$.  The following lemma express the rank gardient of a pro-$p$ group in terms of  a rank of the module $I_{\F_p[[G]]}$.

\begin{lem}\label{rankgr}
Let $G$ be a finitely generated pro-$p$ group. 
Then  $$\rg(G)=\dim_{\F_p[[G]]} I_{\F_p[[G]]}-1+ \frac{1}{|G|}.$$ 
In particular,  if $\Gamma$ is a finitely generated residually-$p$ group, 
then $b_1^{(2)}(\Gamma, \mathbb{F}_p) = \operatorname{rg}(\Gamma_{\hat{p}})$.
\end{lem}
\begin{proof} The case where $G$ is finite is clear. Assume $G$ is infinite. Let  $G>U_1>U_2>\ldots $ be a residual chain of normal open subgroups of $G$. 
Consider the exact sequence
$$0\to I_{\F_p[[G]]}\to \F_p[[G]]\to \F_p\to 0.$$
After applying $\F_p\otimes_{\F_p[[U_i]]}$ we obtain the exact sequence
$$0\to {\rm H}_1(U_i;\F_p)\to \F_p\otimes_{\F_p[[U_i]]}I_{\F_p[[G]]}\to \F_p[G/U_i]\to \F_p\to 0.$$
Therefore we obtain  $$\dim_{\F_p[G/U_i]} I_{\F_p[[G]]}-1=\frac{\dim_{\F_p} {\rm H}_1(U_i;\F_p)}{|G:U_i|}-\frac{1}{|G:U_i|}.$$
  Hence, by Proposition \ref{modpfg}, $$\dim_{\F_p[[G]]} I_{\F_p[[G]]}-1=\rg(G).  $$
  
  Now, let $\Gamma$ be a finitely generated residually-$p$ group. We put $G=\Gamma_{\hat{p}}$ and let $G>U_1>U_2>\ldots $ be a residual chain of normal open subgroups of $G$. Then as before we have that
  $$\dim_{\F_p[G/U_i]} I_{\F_p[\Gamma]}-1=\frac{\dim_{\F_p} {\rm H}_1(U_i;\F_p)}{|G:U_i|}-\frac{1}{|G:U_i|},  $$
   and so $b_1^{(2)}(\Gamma, \mathbb{F}_p) = \operatorname{rg}(G)$.
\end{proof}

The following lemma shows that we can control  $\dim_{\F_p[[G]]} \F_p[[G]]/\F_p[[G]](g-1)$.

\begin{lem}\label{boundpower}
Let $G$ be a countably based pro-$p$ group and $g\in G$ of order $p^k$. Then $\dim_{\F_p[[G]]} \F_p[[G]]/\F_p[[G]](g-1) =\frac  1{p^k}.$ \end{lem}
\begin{proof}
Let  $G>U_1>U_2>\ldots $ be a residual chain of normal open subgroups of $G$. Without loss of generality we can assume that the order of the image of $g$ in $\F_p[G/U_1]$ (and so in $\F_p[G/U_i]$ for all $i$) is $p^k$. Then we obtain that 
\begin{multline*}
\dim_{\F_p[[G]]} \F_p[[G]]/\F_p[[G]](g-1)=\\ \lim_{i\to \infty}\dim_{\F_p[G/U_i]} \F_p[[G]]/\F_p[[G]](g-1) =\lim_{i\to \infty} \frac{ |G:\langle g\rangle U_i|}{|G:U_i|}\\  \lim_{i\to \infty}
\frac{ 1}{|\langle g\rangle U_i:U_i|}=\lim_{i\to \infty}\frac 1{p^k}=\frac 1{p^k}.\qedhere \end{multline*}
\end{proof}

Now we are ready to prove Propostion  \ref{rankgradient}.
\begin{proof}[Proof of Propostion  \ref{rankgradient}]
If the order of $g$ is less than $p^k$, then $G=\overline G$, and so the claim of the theorem follows. If $p^k$ divides the order of $g$ (the order of an element of infinite order in a pro-$p$ group is $p^{\infty}$), then there exists a finite quotient $P$ of $G$, such that the order of the image of $g$ in $P$ is exactly $p^k$. Hence $P$ is also a quotient of $\overline G$, and so
the order of $\overline g$, the image of  $g$ in $\overline G$, is exactly $p^k$.

Let $\widehat \otimes$ denote the profinite tensor product \cite[Chapter 5]{RZ10}. Consider the exact sequence
$$0\to I_{\F_p[[G]]}\to \F_p[[G]]\to \F_p\to 0.$$
After applying $\F_p[[\overline G]]\widehat \otimes_{\F_p[[G]]}$, we obtain the exact sequence
$$0\to {\rm }H_1(G,\F_p[[\overline G]])\to  \F_p[[\overline G]]\widehat \otimes_{\F_p[[G]]} I_{\F_p[[G]]}\xrightarrow{\alpha} \F_p[[\overline G]]\to \F_p\to 0.$$
By Shapiro's Lemma \cite[Theorem 6.10.8]{RZ10}, $$H_1(G,\F_p[[\overline G]])\cong H_1(N,\F_p)\cong N/\overline{[N,N]N^p}.$$ Since $N$ is the normal closed subgroup generated by $g^{p^k}$, 
the $\F_p[[\overline G]]$-module $ N/\overline{[N,N]N^p}$ is a quotient of $\F_p[[\overline G]]/\F_p[[\overline G]](\overline g-1)$. Taking into account that $\im \alpha=I_{\F_p[[\overline G]]}$, we obtain an exact sequence
$$\F_p[[\overline G]]/\F_p[[\overline G]](\overline g-1)\to  \F_p[[\overline G]]\widehat \otimes_{\F_p[[G]]} I_{\F_p[[G]]}  \to I_{\F_p[[\overline G]]}\to 0.$$
 Let  $G>U_1>U_2>\ldots $ be a residual chain of normal open subgroups of $G$.  By Proposition \ref{modpfg}, 
\begin{multline*}
\dim_{\F_p[[\overline G]]}  \F_p[[\overline G]]\widehat \otimes_{\F_p[[G]]} I_{\F_p[[G]]} \myeq{Prop. \ref{modpfg} (1)}\\
\lim_{i\to \infty} \dim_{\F_p[G/NU_i]} \F_p[G/NU_i]\otimes_{\F_p[[\overline G]]} (\F_p[[\overline G]]\widehat \otimes_{\F_p[[G]]} I_{\F_p[[G]]})=\\
\lim_{i\to \infty} \dim_{\F_p[  G/NU_i]} \F_p[G/NU_i]\otimes_{\F_p[[ G]]}  I_{\F_p[[G]]}\myeq{Prop. \ref{modpfg} (1)}\\\dim_{\F_p[[\overline G]]}  I_{\F_p[[G]]} \myge{Prop. \ref{modpfg} (2)} 
\dim_{\F_p[[ G]]}  I_{\F_p[[G]]}
\end{multline*}
On the other hand, by Lemma \ref{boundpower}, 
$$\dim_{\F_p[[\overline G]]}\F_p[[\overline G]]/\F_p[[\overline G]](\overline g-1)=\frac 1{p^k}.$$
Hence, from (SMod3), we obtain that
$$\dim_{\F_p[[\overline G]]} I_{\F_p[[\overline G]]}\ge \dim_{\F_p[[ G]]}  I_{\F_p[[G]]}-\frac 1{p^k}.$$
By Lemma \ref{rankgr} this means that  $\rg(\overline G)\ge \rg(G)-\frac 1 {p^k}$.
 \end{proof}

Proposition  \ref{rankgradient} implies the following result of Schlage-Puchta \cite{SP12}.
 
\begin{cor} \cite[Theorem 1]{SP12} Let $d\ge 2$ be an integer, $p$ a prime, and $\epsilon>0$. Then there exists a
$d$-generated   $p$-group $\Gamma$ with $\rg(\Gamma_{\hat p})\ge d-1-\epsilon$.
\end{cor}
\begin{proof}
Let $F$ be a free group of rank $d$. List all non-trivial elements of $F$: $\{f_1,f_2,f_3,\ldots \}$. Take $k_i\in \N$ such that 
$\displaystyle \sum_{i=1}^{\infty} \frac 1{p^{k_i}}\le \epsilon$. We put $$\Gamma_n=F/\langle \langle f_1^{p^{k_1}},\ldots, f_k^{p^{k_n}} \rangle \rangle \textrm{\ and\ }\Gamma=F/\langle \langle f_i^{p^{k_i}}| i\in \N \rangle \rangle .$$ It is clear that $\Gamma$ is a $p$-group. We claim that $\rg(\Gamma_{\hat p})\ge d-1-\epsilon$.

In view of (\ref{infimum2}) it is enough to show that if $N$ is a normal subgroup of $\Gamma$ of finite index (which is a power of $p$ because $\Gamma$ is a $p$-group), then $\frac{\dim_{\F_p} {\rm H}_1(N;\F_p)}{|\Gamma:N|}\ge d-1-\epsilon$. Denote by $\pi:F\to \Gamma$ and $\pi_n:F\to \Gamma_n$ the canonical maps and let $\widetilde N=\pi^{-1}(N)$. Since $ [\widetilde N, \widetilde N]\widetilde N^p$ is of finite index in $F$, there exists $n\in \N$ such that $f_i^{p^{k_i}}\in [\widetilde N, \widetilde N]\widetilde N^p$ for all $i>n$. Therefore,
$$\frac{\dim_{\F_p} {\rm H}_1(N;\F_p)}{|\Gamma:N|}=\frac{\dim_{\F_p} {\rm H}_1(\pi_n(\widetilde N);\F_p)}{|\Gamma_n:\pi_n(\widetilde N)|}\ge \rg((\Gamma_n)_{\hat p}).$$
By Proposition  \ref{rankgradient}, $\rg((\Gamma_n)_{\hat p})\ge d-1-\displaystyle \sum_{i=1}^n \frac 1{p^{k_i}}\ge  d-1-\epsilon$.
\end{proof}

We finish the section by  proving Propostion  \ref{EL}.

\begin{proof}[Proof of Proposition \ref{EL}]
By the formula (\ref{l2}), 
$b_1^{(2)}(\Gamma, \Q)=\dim_{\Q[\Gamma]} I_{\Z[\Gamma]}-1+\frac 1{|\Gamma]}$ and by Lemma  \ref{rankgr}
$\rg(\Gamma_{\hat p})=b_1^{(2)}(\Gamma, \F_p)=\dim_{\F_p[[\Gamma_{\hat p}]]}I_{\Z[\Gamma]}-1+\frac 1{|\Gamma]}.$ 

Let  $\Gamma_{\hat p}>U_1>U_2>\ldots $ be a residual chain of normal open subgroups of $\Gamma_{\hat p}$.  Put $N_i=\Gamma \cap U_i$.
Then, from the definitions of $\dim_{\Q[\Gamma]}$ and $\dim_{\F_p[[\Gamma_{\hat p}]]}$ as elements of $\mathbb P(\Z[\Gamma])$, we obtain
$$
\dim_{\Q[\Gamma]}=\lim_{i\to \infty} \dim_{\Q[\Gamma/N_i]}\le \lim_{i\to\infty} \dim_{\F_p[\Gamma/N_i]}=\dim_{\F_p[[\Gamma_{\hat p}]]}.
$$

Thus, $\dim_{\Q[\Gamma]} I_{\Z[\Gamma]}\le \dim_{\F_p[[\Gamma_{\hat p}]]}I_{\Z[\Gamma]}$, and so $b_1^{(2)}(\Gamma, \Q)\le \rg (\Gamma_{\hat p})$.
\end{proof}
\bibliographystyle{amsalpha}
\bibliography{biblio}

\end{document}